\documentclass[10pt]{amsart}

\usepackage{graphicx}
\usepackage{amsmath}
\usepackage{amssymb}

\newtheorem{theorem}{Theorem}[section]
\newtheorem{proposition}{Proposition}[section]
\newtheorem{lemma}{Lemma}[section]

\theoremstyle{definition}

\newtheorem{example}[theorem]{Example}

\theoremstyle{remark}
\newtheorem{remark}[theorem]{Remark}
\newtheorem{corollary}[theorem]{Corollary}

\numberwithin{equation}{section}

\newcommand{\ldim}{\underline{\dim}_{B}}
\newcommand{\udim}{\overline{\dim}_{B}}
\newcommand{\UDI}{\mathfrak{Di}}
\newcommand{\LDI}{\mathfrak{di}}
\newcommand{\bt}{\boldsymbol{t}}

\newcommand{\btau}{\boldsymbol{\tau}}
\newcommand{\bq}{\boldsymbol{q}}
\newcommand{\BR}{\mathbb{R}}
\newcommand{\BZ}{\mathbb{Z}}
\newcommand{\BT}{\mathbb{T}}
\newcommand{\BG}{\mathbb{G}}
\newcommand{\BQ}{\mathbb{Q}}
\newcommand{\MX}{\mathcal{X}}
\newcommand{\MY}{\mathcal{Y}}
\newcommand{\MT}{\mathcal{T}}
\newcommand{\MO}{\mathcal{O}}
\newcommand{\MK}{\mathcal{K}}
\newcommand{\HULL}{\mathcal{H}}
\newcommand{\apf}{\phi}
\newcommand{\apff}{\hat{\phi}}
\newcommand{\apfNEW}{\varphi}
\newcommand{\apfh}{\overline{\varphi}}

\begin{document}

\title[Linear Inhomogeneous Approximations]
{On the Dimensional-like Characteristics Arising From Linear Inhomogeneous Approximations}

\author{Mikhail Anikushin}
\address{Department of
	Applied Cybernetics, Faculty of Mathematics and Mechanics,
	Saint-Petersburg State University, Saint-Petersburg, Russia.}
\email{demolishka@gmail.com}
\thanks{This work is supported by the German-Russian
	Interdisciplinary Science Center (G-RISC) funded by the German Federal
	Foreign Office via the German Academic Exchange Service (DAAD) (Project M-2017a-5, Project M-2017b-9).}



\keywords{Kronecker theorem, Almost periodic function, Dimension theory, Diophantine approximation, Diophantine condition}

\begin{abstract}
As it follows from the theory of almost periodic functions the set of integer solutions $q$ to the Kronecker system $|\omega_{j} q - \theta_{j}| < \varepsilon \pmod 1$, $j=1,\ldots,m$, where $1,\omega_{1},\ldots,\omega_{m}$ are linearly independent over $\BQ$, is relatively dense in $\BR$. The latter means that there exists $L(\varepsilon)>0$ such that any segment of the length $L(\varepsilon)$ contains at least one integer solution to the Kronecker system. We give lower and upper estimate for $L(\varepsilon)$ and show that $L(\varepsilon) = \left(\frac{1}{\varepsilon}\right)^{m+o(1)}$ as $\varepsilon \to 0$ for many cases, including algebraic numbers as well as badly approximable numbers. We use methods of dimension theory and Diophantine approximations of $m$-tuples satisfying Diophantine condition.
\end{abstract}

\maketitle

\section{Introduction}
The Kronecker theorem states that if $1,\omega_1,\ldots,\omega_m$ are linearly independent over rationals then for every $\varepsilon>0$ and $\theta_1,\ldots,\theta_m \in \BR$ the Kronecker system $|\omega_j q - \theta_{j}| < \varepsilon \pmod 1$, $j=1,\ldots,m$ has an integer solution $q$. One may ask what is an upper or lower bound for such $q$ (more precisely, for the absolute value of the first integer solution $q$ closest to zero) in terms of $\varepsilon, m$ and some properties of $\omega_{1},\ldots,\omega_{m}$? There are papers where the so called effective upper bounds for $|q|$ are given (see \cite{fukshansky2018effective,Vorselen2010} and links therein). Usually, such bounds are given under the consideration of algebraic numbers $\omega_1,\ldots,\omega_m$ and, therefore, powerful methods of algebraic number theory (see, for example, \cite{cassels1986,lang1994algebraic,Schmidt1980}) are used. A typical bound is $|q| \leq C \left(\frac{1}{\varepsilon}\right)^{d-1}$, where $d=\left[ \BQ(\omega_{1},\ldots,\omega_{m}) \colon \BQ \right]$ and the constant $C$ depends on $m$, $d$ and the heights and degrees of $\omega_{1},\ldots,\omega_{m}$. The effectiveness of a bound means that the constant $C$ can be directly calculated. It is well-known that if one removes the requirement of effectiveness, the exponent $d-1$ may be changed to (the stronger one) $m+\delta$ for any $\delta>0$ (see, for example, remark 3.1 in \cite{fukshansky2018effective} or Theorem 2.1 in \cite{Moser1990}).

On the other hand, as the theory of almost periodic functions (see \cite{pankov2012}) says, the set of integer solutions to the Kronecker system are relatively dense, namely, there is $L(\varepsilon)>0$ such that every segment of length $L(\varepsilon)$ contains an integer solution. Now one may ask: what are possible lower or upper bounds for $L(\varepsilon)$ or what is the growth rate of $L(\varepsilon)$ as $\varepsilon$ tends to zero\footnote{More precisely, here by $L(\varepsilon)$ we mean the best possible, i. e. the infimum, of all such values. For details, see below.}? Here we use a dimension theory approach (see \cite{AnikushinReitmann2016, LeoKuzReit2017, Naito1996}) combined with Diophantine approximations (see \cite{Khinchin1997,Kleinblock1998,Schmidt1980}) to provide some lower and upper bounds for $L(\varepsilon)$. Despite that these bounds are ineffective, we get additional information (for example, about the distribution of such solutions and exact values of dimensional-like characteristics) and treat a more general than just an algebraic set of $\omega$'s, providing a different view on the problem. This complements some known results, namely, effective versions (obtained via algebraic number theory \cite{fukshansky2018effective,Vorselen2010}) and quantitative versions (derived from transference principles \cite{cassels1957}) of the Kronecker theorem.
To state our results precisely, we need some concepts.
\subsection*{Diophantine dimension}
A subset $\mathcal{R} \subset \BR^{n}$ is called \textit{relatively dense} in $\BR^{n}$ if there is a real number $L>0$ such that the set $(a + [0,L]^{n}) \cap \mathcal{R}$ is not empty for all $a \in \BR^{n}$.

Let $\mathfrak{R} = \{\mathcal{R}_{\varepsilon}\}$, $\varepsilon>0$, be a family of relatively dense in $\BR^{n}$ subsets $\mathcal{R}_{\varepsilon} \subset \BR^{n}$ such that $\mathcal{R}_{\varepsilon_{1}} \supset \mathcal{R}_{\varepsilon_{2}}$ provided by $\varepsilon_{1} > \varepsilon_{2}$. Let $L(\varepsilon)>0$ be a real number such that $(a+[0,L(\varepsilon)]^{n}) \cap \mathcal{R}_{\varepsilon}$ is not empty for all $a \in \BR^{n}$. Let $l_{\mathfrak{R}}(\varepsilon)$ be the infimum of all such $L(\varepsilon)$. Then $l_{\mathfrak{R}}(\varepsilon)$ is the \textit{inclusion length} for $\mathcal{R}_{\varepsilon}$. The value\footnote{In the case $l_{\mathfrak{R}}(\varepsilon)=0$ for all small $\varepsilon$ one can use $\ln(l_{\mathfrak{R}}(\varepsilon)+1)$ instead of $\ln l_{\mathfrak{R}}(\varepsilon)$. Anyway, in this case $\UDI(\mathfrak{R})=0$ and the case is out of our interest.}
\begin{equation}
\UDI(\mathfrak{R}):=\limsup\limits_{\varepsilon \to 0+} \frac{\ln l_{\mathfrak{R}}(\varepsilon)}{\ln \left(1/\varepsilon\right)}
\end{equation}
is called the \textit{Diophantine dimension} of $\mathfrak{R}$. Also we consider the \textit{lower Diophantine dimension} of $\mathfrak{R}$ defined as
\begin{equation}
\LDI(\mathfrak{R}):=\liminf\limits_{\varepsilon \to 0+} \frac{\ln l_{\mathfrak{R}}(\varepsilon)}{\ln \left(1/\varepsilon\right)}.
\end{equation}

\subsection*{Box-counting dimension}
Let $\MX$ be a compact metric space and let $N_{\varepsilon}(\MX)$ denote the minimal number of open balls of radius $\varepsilon$ required to cover $\MX$. The values
\begin{equation}
\begin{split}
\ldim\MX &= \liminf_{\varepsilon \to 0+}\frac{\ln N_{\varepsilon}(\MX)}{\ln (1/\varepsilon)},\\
\udim\MX &= \limsup_{\varepsilon \to 0+}\frac{\ln N_{\varepsilon}(\MX)}{\ln (1/\varepsilon)}
\end{split}
\end{equation}
are called \textit{lower box dimension} and \textit{upper box dimension} respectively.

\subsection*{Diophantine condition}
For $\theta \in \BR^{m}$ we denote by $|\theta|_{m}$ the distance from $\theta$ to $\BZ^{m}$. Clearly, $|\cdot|$ defines a metric on $m$-dimensional flat torus $\BT^{m}=\BR^{m} / \BZ^{m}$.

We say that an $m$-tuple $\omega=(\omega_{1},\ldots,\omega_{m})$ of real numbers satisfy the \textit{Diophantine condition} of order $\nu \geq 0$ if for some $C_{d}>0$ and all natural $q$ the inequality
\begin{equation}
|\omega q|_{m} \geq C_{d}\left(\frac{1}{q}\right)^{\frac{1+\nu}{m}}
\end{equation}
holds.

For a function $\apf \colon \BR^{n} \to \MX$ let $\MO(\apf)$ be the closure of $\bigcup\limits_{\bt \in \BR^{n}}\apf(\bt)$ in $\MX$ and let $\mathring{\MO}(\apf)$ be the closure of $\bigcup\limits_{\bq \in \BZ^{n}}\apf(\bq)$ in $\MX$.

\begin{theorem}
	\label{TH: MainResMatrixTh}
	Let $A$ be an $(m \times n)$-matrix with real coefficients and $\theta \in \mathring{\MO}(\apf)$, where $\apf \colon \BR^{n} \to \BT^{m}$ is defined by $\apf(\bt):=A\bt$; then the set $\mathcal{R}_{\varepsilon}$ of integer solutions $\bq \in \BZ^{n}$ to
	\begin{equation}
	\label{EQ: MatrixKroneckerSysMainRes}
	|A\bq - \theta|_{m} < \varepsilon
	\end{equation}
	is relatively dense in $\BR^{n}$ and the lower Diophantine dimension of $\mathfrak{R}=\{ \mathcal{R}_{\varepsilon} \}$ satisfies
	\begin{equation}
	\label{EQ: MainResEstimateLower}
	\LDI(\mathfrak{R}) \geq \frac{d-n}{n},
	\end{equation}
	where $d=\ldim\MO(\apff)$ and $\apff \colon \BR^{n} \to \BT^{m+n}$ is defined by $\apff(\bt):=\hat{A}\bt$, $\hat{A}= \begin{bmatrix}
	E \\
	A
	\end{bmatrix}.$ 
\end{theorem}
\noindent The proof of theorem \ref{TH: MainResMatrixTh} is given at the end of section 3.

\begin{theorem}
	\label{TH: MainResUpperEstimate}
	Suppose that $1,\omega_{1},\ldots,\omega_{m}$ are linearly independent over $\BQ$ and let $\omega=(\omega_{1},\ldots,\omega_{m})$ satisfy the Diophantine condition of order $\nu \geq 0$ such that $\nu (m-1) < 1$; then for all $\theta_{1},\ldots,\theta_{m} \in \BR$ and $\varepsilon>0$ the set $\mathcal{K}_{\varepsilon}$ of integer solutions $q$ to
	\begin{equation}
	\label{EQ: MainResultKroneckerBase}
	|\omega q - \theta |_{m} < \varepsilon,
	\end{equation}
	is relatively dense in $\BR$ and the Diophantine dimension $\UDI(\mathfrak{K})$ of $\mathfrak{K}=\{ \mathcal{K}_{\varepsilon} \}$ satisfies the inequality
	\begin{equation}
	\label{EQ: MainResEstimateUpper}
	\UDI(\mathfrak{K}) \leq \frac{(1+\nu)m}{1-\nu(m-1)}.
	\end{equation}
\end{theorem}
\noindent The proof of theorem \ref{TH: MainResUpperEstimate} is given at the end of section 4.

Estimates \eqref{EQ: MainResEstimateLower} and \eqref{EQ: MainResEstimateUpper} can be used to prove the following corollary (the proof is outlined in section 5).
\begin{corollary}
	\label{COR: ExactCalculation}
	There is a set of full measure $\Omega_{m} \subset \BR^{m}$ such that for any $\omega=(\omega_{1},\ldots,\omega_{m}) \in \Omega_{m}$ the Diophantine dimensions of $\mathfrak{K}$ (see the previous theorem) satisfy
	\begin{equation}
	\LDI(\mathfrak{K})=\UDI(\mathfrak{K})=m.
	\end{equation}
	In particular, badly approximable and algebraic $m$-tuples, which are linearly independent over $\BQ$, satisfy the above conditions.
\end{corollary}
\noindent Thus, within the assumptions of corollary \ref{COR: ExactCalculation}, for all small $\delta>0$ and for some ineffective constants $C^{+}(\delta)$ and $C_{-}(\delta)$, we have an integer solution to \eqref{EQ: MainResultKroneckerBase} in each interval of length $C^{+}(\delta) \left(\frac{1}{\varepsilon}\right)^{m+\delta}$ and there are gaps of length $C_{-}(\delta)\left(\frac{1}{\varepsilon}\right)^{m-\delta}$ with no integer solutions.

The main idea of our approach is as follows. To study integer solutions $\bq \in \BZ^{n}$ to \eqref{EQ: MatrixKroneckerSysMainRes} (the \textit{discrete problem}) we consider the extended system with the matrix 
\begin{equation}
\label{EQ: ExtendedMatrix}
\hat{A}= \begin{bmatrix}
E \\
A
\end{bmatrix}
\end{equation}
and with respect to $\bt \in \BR^{n}$ (the \textit{continuous problem}). By the choice of $\hat{A}$, any solution $\bt$ of the extended system is close to $\BZ^{m}$ and, therefore, the corresponding Diophantine dimensions of the discrete problem and the continuous problem coincide (see proposition \ref{PROP: ConttoDiscrete}). Due to the linearity of $\hat{A}$ (which is essential) the corresponding Diophantine dimension of the family of $\varepsilon$-solutions to the extended system is equal to the Diophantine dimension of the almost periodic function $\apff(\bt):=\hat{A}\bt$ defined by the family of $\varepsilon$-almost periods of $\apff(\cdot)$ (see proposition \ref{PROP: LinearityDimensionCoincide}). Thus, it is enough to study the Diophantine dimension of $\apff$. For the latter purpose we will use developed methods from our earlier works \cite{Anikushin2017, Anikushin2016}.

A research interest in such properties, in addition to the purely algebraic one, may come from almost periodic dynamics (see \cite{Anikushin2017,Anikushin2016,Katok1997,Naito1996}). For example, it is well-known that some number-theoretical phenomena appear in the linearization of circle diffeomorphisms as well as in KAM theory (see \cite{Katok1997, Moser1990}). As it follows from the arithmetical nature of almost periods there is a strong connection between them and the Diophantine approximations of the Fourier exponents. More precisely, the latter affects the growth rate of the inclusion length. To study such a connection, a definition of Diophantine dimension was given (see \cite{Anikushin2017}). A method for upper estimates of the inclusion length (and, therefore, of the Diophantine dimension) firstly appeared in \cite{Naito1996} for the case of badly approximable numbers. In \cite{Anikushin2017} it was generalized for quasi-periodic functions with one irrational frequency that satisfies the Diophantine condition. In the present paper we generalize such an approach (theorem \ref{TH: UpperDiophantineEstimateTheoremMULTI}) to give an upper estimate of the Diophantine dimension for quasi-periodic functions with frequency $m$-tuple, satisfying simultaneous Diophantine condition. A dimensional argument (as in theorem \ref{TH: LowerEstimate}) to provide a lower estimate of the Diophantine dimension firstly appeared in \cite{Anikushin2016}, where the recurrence properties of almost periodic dynamics were studied (as well as in \cite{Naito1996}).

This paper is organized as follows. In section 2 we begin with some basic definitions from the theory of almost periodic functions. Next, we show how the discrete problem is simply connected  with the continuous one. In section 3 we give a lower bound for the Diophantine dimension (theorem \ref{TH: LowerEstimate}), using a dimensional argument, namely, we use the lower box dimension of the orbit closure. In section 4 we present an upper bound for the case of $m$-tuple, satisfying simultaneously Diophantine condition (theorem \ref{TH: UpperDiophantineEstimateTheoremMULTI}), for what we need a proper sequence of simultaneous denominators (=convergents), provided by theorem \ref{TH: DiophantineMultiCondition}. Section 5 is devoted to the discussion of the presented approach and its consequences, in particular, concerning the Kronecker theorem.
\section{Preliminaries}
\subsection*{Almost periodic functions}

Let $\BG$ be a locally compact abelian group and let $\MX$ be a complete metric space endowed with a metric $\varrho_{\MX}$. A continuous function $\apf \colon \BG \to \MX$ is called \textit{almost periodic}\footnote{More precisely, such functions are called uniformly almost periodic or Bohr almost periodic.} if for every $\varepsilon>0$ the set $\MT_{\varepsilon}(\apf)$ of $\tau \in \BG$ such that
\begin{equation}
\varrho\left(\apf(\cdot + \tau), \apf(\cdot)\right)_{\infty}:=\sup_{t \in \BG}\varrho_{\MX}(\apf(t+\tau),\apf(t)) \leq \varepsilon
\end{equation}
is relatively dense in $\BG$, i. e. there is a compact set $\MK=\MK(\varepsilon) \subset \BG$ such that $\left(g+\MK\right) \cap \MT_{\varepsilon}(\apf) \not= \emptyset$ for all $g \in \BG$. Here $\tau \in \MT_{\varepsilon}(\apf)$ is called an \textit{$\varepsilon$-almost period} of $\apf$. The following theorem is due to Bochner (see theorem 1.2 and remark 1.4 in \cite{pankov2012}).
\begin{theorem}
	\label{TH: Bochner TH}
	A bounded continuous function $\apf(\cdot)$ is almost periodic if and only if every sequence $\{ \apf(\cdot + \tau_{n}) \}$, $\tau_{n} \in \BG, n=1,2,\ldots$, contains a uniformly convergent subsequence.
\end{theorem}
\begin{example}
	\label{EX: MainSpacesAndFunctions}
	Let $\BG = \BR^{n}$ and $\MX=\BT^{m}$, where $\BT^{m}:=\BT^{m}/\BZ^{m}$ is $m$-dimensional flat torus. For $\theta \in \BT^{m}$ let $|\theta|_{m}$ be the distance from $\theta$ (more formally, from any representative of $\theta$) to $\BZ^{m}$. Then a metric on $\BT^{m}$ is given by
	\begin{equation}
	\varrho_{m}(\theta_{1},\theta_{2}):=|\theta_{1}-\theta_{2}|_{m}.
	\end{equation}
	Note that the product $\theta t$ for $\theta \in \BT^{m}$ and $t \in \BR$ is well-defined, as well as any function $f \colon \BR^{s} \to \BR^{m}$ can be considered as a function $f \colon \BR^{s} \to \BT^{m}$.
	
	In this case it is quite clear that continuous periodic functions are almost periodic (by definition) and any sum of continuous periodic functions is almost periodic (by the Bochner theorem). Consider $\apf \colon \BR^{n} \to \BT^{m}$ defined as $\apf(\bt):=A\bt$, where $A$ is an $m \times n$ matrix with real coefficients. It is clear that the function $\apf(\cdot)$ is almost periodic due to its additivity and compactness of $\BT^{m}$.
	
	A more simple example (also known as a linear flow on $\BT^{m}$) appears when $n=1$ and $\apf(t):=(\omega_1 t,\ldots, \omega_m t)$, where $\omega_1,\ldots,\omega_m \in \BR$.
\end{example}

\subsection*{Basic constructions}

Further (see propositions \ref{PROP: LinearityDimensionCoincide} and \ref{PROP: ConttoDiscrete}) we will have to show some relations between the Diophantine dimensions of two families $\mathfrak{R}'=\{ \mathcal{R}'_{\varepsilon}\}$ and $\mathfrak{R}''=\{ \mathcal{R}''_{\varepsilon}\}$. Note that in order to show the inequality $\UDI(\mathfrak{R}') \leq \UDI(\mathfrak{R}'')$ it is sufficient to show the inclusion $\left(\mathcal{R}''_{C \varepsilon} + \bt_{0}(\varepsilon)\right) \subset \mathcal{R}'_{\varepsilon}$ for some constant $C>0$ and $\bt_{0}(\varepsilon) \in \BR^{n}$.

We will deal with the case when $\mathcal{R}_{\varepsilon}$ is a set of integer solutions to the Kronecker system or, more generally, the set of moments of return (integer or real) in $\varepsilon$-neighbourhood of a point in the closure of almost periodic trajectory.

Let $\BG = \BR^{n}$ and let $\MX$ be a complete metric space. Consider a non-constant almost periodic function $\apf \colon \BR^{n} \to \MX$. By definition, the set $\MT_{\varepsilon}(\apf)$ of $\varepsilon$-almost periods of $\apf$ is relatively dense. We use the notations $\UDI(\apf)$ and $\LDI(\apf)$ for the Diophantine dimensions of the corresponding family $\{ \MT_{\varepsilon}(\apf) \}$ and we call them the \textit{Diophantine dimension} of $\apf$ and the \textit{lower Diophantine dimension} of $\apf$ respectively. Note that the set $\MT_{\varepsilon}(\apf) \cap \BZ^{n}$ is relatively dense too\footnote{To prove that consider an almost periodic function with values in $\MX \times \BT^{m}$ defined as $\bt \mapsto (\apf(\bt),\bt)$. Almost periods of such a function are <<almost integers>> and contained in $\MT_{\varepsilon}(\apf)$. Using the uniform continuity argument, one can show that such an almost period can be slightly perturbed to become an integer.}. We use the notations $\mathring{\UDI}(\apf)$ and $\mathring{\LDI}(\apf)$ for the Diophantine dimension and the lower Diophantine dimension of the corresponding family $\{\MT_{\varepsilon}(\apf) \cap \BZ^{n}\}$.

Recall, that we use the notations $\MO(\apf)$ and  $\mathring{\MO}(\apf)$ for the closure (in $\MX$) of $\apf(\BR^{n})$ and $\apf(\BZ^{n})$ respectively. The following example shows that these sets can differ.
\begin{example}
	Let $\phi \colon \BR^{2} \to \BT^{2}$ and $\phi(\bt)=(t_{1},\sqrt{2}t_{2})$. Here $\MO(\phi)$ is entire $\BT^{2}$ and $\mathring{\MO}(\phi)$ is just a segment.
\end{example}

For $\varepsilon>0$ and $x \in \MO(\apf)$ consider the system of inequalities with respect to $\bt \in \BR^{n}$:
\begin{equation}
\label{EQ: GeneralAPSystem}
\varrho_{\MX}(\apf(\bt),x) \leq \varepsilon.
\end{equation}
It is easy to see that if $\bt=\bt_{0}$ is a solution to \eqref{EQ: GeneralAPSystem} and $\tau \in \MT_{\varepsilon}(\apf)$ then $\bt_{0} + \tau$ is a $2\varepsilon$-solution (i. e. with $\varepsilon$ changed to $2\varepsilon$) to \eqref{EQ: GeneralAPSystem}. Thus, the set of solutions to \eqref{EQ: GeneralAPSystem} is relatively dense. We use the notations $\UDI(\phi;x)$ and $\LDI(\phi;x)$ for the corresponding Diophantine dimensions. Note that the given argument provides the inequalities $\UDI(\apf;x) \leq \UDI(\apf)$ and $\LDI(\apf;x) \leq \LDI(\apf)$.

For $\varepsilon>0$ and $x \in \mathring{\MO}(\apf)$ consider the system of inequalities with respect to $\bq \in \BZ^{n}$:
\begin{equation}
\label{EQ: GeneralAPSystemInteger}
\varrho_{\MX}(\apf(\bq),x) \leq \varepsilon.
\end{equation}
As in the previous case, one can show that the set of solutions to \ref{EQ: GeneralAPSystemInteger} is relatively dense. Here we use the notations $\mathring{\UDI}(\phi;x)$ and $\mathring{\LDI}(\phi;x)$ for the corresponding Diophantine dimensions.

Now for $\apf \colon \BR^{n} \to \BT^{m}$ defined by $\apf(\bt)=A\bt$ (as in example \ref{EX: MainSpacesAndFunctions}) and $\theta \in \MO(\apf)$ consider the system
\begin{equation}
\label{EQ: TorusLinearSystem}
|A\bt - \theta|_{m} \leq \varepsilon.
\end{equation}
\begin{proposition}
	\label{PROP: LinearityDimensionCoincide}
	For $\apf(\cdot)$ defined above we have
	\begin{equation}
	\begin{split}
	\UDI(\apf;\theta)&=\UDI(\apf),\\
	\LDI(\apf;\theta)&=\LDI(\apf).
	\end{split}
	\end{equation}
\end{proposition}
\begin{proof}
	Let $\bt_{0}$ be a fixed solution to $\eqref{EQ: TorusLinearSystem}$ and let $\tau$ be an arbitrary solution to $\eqref{EQ: TorusLinearSystem}$. For $\tau'=\tau-\bt_{0}$ we have
	\begin{equation}
	|A(\tau' + \bt) - A(\bt)|_{m} = |A\tau - A\bt_{0}|_{m} \leq |A\tau - \theta|_{m} + |A\bt_{0}-\theta|_{m} \leq 2\varepsilon.
	\end{equation}
	Thus, $\tau' \in \MT_{2\varepsilon}(\apf)$ and, consequently, $\UDI(\phi;\theta) \geq \UDI(\apf)$ and $\LDI(\apf;\theta) \geq \LDI(\apf)$. The inverse inequalities were shown before.
\end{proof}
\noindent Additivity of $\phi(\cdot)$ plays a central role in the proof of proposition \ref{PROP: LinearityDimensionCoincide}. Consider the following
\begin{example}
	Let $\apf \colon \BR \to \BT^{2}$ be given by $\apf(t):=\left(\sin (2\pi t) \sin(2\pi \sqrt{2} t),\cos(2\pi t)\right)$. It is clear that every integer is a solution to the system $|\apf(\bt)|_{m} < \varepsilon$ (i. e. for $\theta = 0$). Thus, $\UDI(\apf;0)=0$, but as it will be shown below $\UDI(\apf)=1$.
\end{example}

Our purpose is to study the set of integer solutions to \eqref{EQ: TorusLinearSystem}, i.e. solutions $\bq \in \BZ^{n}$ for the system (we call it also the \textit{Kronecker system})
\begin{equation}
\label{EQ: LinearSysInteger}
|A\bq - \theta|_{m} \leq \varepsilon,
\end{equation}
where $\theta \in \mathring{\MO}(\apf) \subset \BT^{n}$. For a transition to the continuous problem we consider the extended system with respect to $\bt \in \BR^{n}$:
\begin{equation}
\label{EQ: KroneckerSysAdd}
|\hat{A}\bt - \hat{\theta}|_{M} \leq \varepsilon,
\end{equation}
where $M=n+m$, $\hat{\theta}=(0,\ldots,0,\theta_{1},\ldots,\theta_{m}) \in \BT^{M}$ and $\hat{A}$ is the $(M \times n)$-matrix defined in \eqref{EQ: ExtendedMatrix}.
In other words, we are looking for real solutions $\bt$ of \eqref{EQ: LinearSysInteger}, satisfying $n$ additional conditions: $|t_{1}|_{1} \leq \varepsilon,\ldots,|t_{n}|_{1} \leq \varepsilon$. It is clear that an integer solution to \eqref{EQ: LinearSysInteger} is also a solution to \eqref{EQ: KroneckerSysAdd}. As we said the set of solutions to \eqref{EQ: LinearSysInteger} is relatively dense. Let $\apff(\bt):=\hat{A}\bt$ be the corresponding almost periodic function (see example \ref{EX: MainSpacesAndFunctions}). We have the following
\begin{proposition}
	\label{PROP: ConttoDiscrete}
	For the set of solutions to \eqref{EQ: LinearSysInteger} we have
	\begin{equation}
	\begin{split}
	\mathring{\UDI}(\apf;\theta) &= \UDI(\apff),\\
	\mathring{\LDI}(\apf;\theta) &= \LDI(\apff).
	\end{split}
	\end{equation}
\end{proposition}
\begin{proof}
	It is clear that $\apff=\hat{A}\bt$ is Lipschitz continuous, i. e. there is $C>0$ such that $|\hat{A}\bt'-\hat{A}\bt''|_{M}\leq \varepsilon$ provided by $|\bt'-\bt''| \leq C\varepsilon$. Let $\tau$ be a $\frac{C}{2}\varepsilon$-solution to \eqref{EQ: KroneckerSysAdd}. It follows that there is $\bq \in \BZ^{n}$ such that $|\bq - \tau| \leq \frac{C}{2}\varepsilon$ and, consequently, $|\hat{A}\bq-\hat{A}\btau|_{M} \leq \varepsilon$. Thus,
	\begin{equation}
	|A\bq - \theta|_{m} \leq |A\tau - \theta|_{m} + |A\tau -A\bq|_{m} \leq 2\varepsilon.
	\end{equation} 
	We showed that $\mathring{\UDI}(\apf;\theta) \leq \UDI(\apff)$ and $\mathring{\LDI}(\apf;\theta) \leq \LDI(\apff)$. The inequalities $\mathring{\UDI}(\apf;\theta) \geq \UDI(\apff)$ and $\mathring{\LDI}(\apf;\theta) \geq \LDI(\apff)$ were shown before.
\end{proof}
A similar reasoning shows that $\UDI(\apff)=\mathring{\UDI}(\apff)$.

So, the set of solutions to the Kronecker system \eqref{EQ: LinearSysInteger} is relatively dense and to study its Diophantine dimension, it is sufficient, by Propositions \ref{PROP: LinearityDimensionCoincide} and \ref{PROP: ConttoDiscrete}, to study the Diophantine dimension of the almost periodic function $\apff(\cdot)$, corresponding to the extended system \eqref{EQ: KroneckerSysAdd}. 
\section{A lower estimate via dimension theory}
Let $\apfNEW \colon \BR^{n} \to \MX$ be a non-constant almost periodic function. In particular, $\apfNEW$ is uniformly continuous. Let $\delta(\varepsilon)$ be such that $\varrho_{\MX}(\apfNEW(\bt_{1}),\apfNEW(\bt_{2})) \leq \varepsilon$ provided by $\|\bt_{1}-\bt_{2}\|_{\infty} \leq \delta(\varepsilon)$, where $\|\cdot\|_{\infty}$ is the sup-norm in $\BR^{n}$. Let $\delta^{*}(\varepsilon)$ be the supremum of such numbers $\delta(\varepsilon)$. Consider the values
\begin{equation}
\begin{split}
\overline{\Delta}(\apfNEW)&:=\limsup\limits_{\varepsilon \to 0+} \frac{\ln \delta^{*}(\varepsilon)}{\ln \varepsilon},\\
\underline{\Delta}(\apfNEW)&:=\liminf\limits_{\varepsilon \to 0+} \frac{\ln \delta^{*}(\varepsilon)}{\ln \varepsilon}.
\end{split}
\end{equation}

We say that a map $\chi \colon \MX \to \MY$, where $\MX$ and $\MY$ are complete metric spaces, satisfies a \textit{local H\"older condition} with an exponent $\alpha \in (0,1]$ if there are constants $C>0$ and $\varepsilon_{0}>0$ such that $\varrho_{\MY}(\chi(x_{1}),\chi(x_{2})) \leq C \varrho_{\MX}{(x_1,x_2)}^{\alpha}$ provided by $\varrho_{\MX}(x_{1},x_{2}) < \varepsilon_{0}$. 

It is clear that if $\apfNEW(\cdot)$ satisfies a local H\"older condition with an exponent $\alpha \in (0,1]$ then $\underline{\Delta}(\apfNEW) \leq \overline{\Delta}(\apfNEW) \leq \frac{1}{\alpha}$.

Consider the \textit{hull} $\HULL(\apfNEW)$ of an almost periodic function $\apfNEW \colon \BR^{n} \to \MY$, where $\MY$ is a complete metric space, defined as the closure of its translates $\{ \apfNEW(\cdot + \bt) \ | \ \bt \in \BR^{n} \}$ in the uniform norm. By Theorem \ref{TH: Bochner TH}, $\HULL(\apfNEW)$ is a compact subset in the space of bounded continuous functions with the uniform norm. We will estimate the box dimensions of $\HULL(\apfNEW)$ in the following theorem.

\begin{theorem}
	\label{TH: LowerEstimate}
	\begin{equation}
	\label{EQ: MainTheoremofDioDim}
	\begin{split}
	\frac{1}{n} \udim\HULL(\apfNEW) &\leq \UDI(\apfNEW) +\overline{\Delta}(\apfNEW),\\
	\frac{1}{n} \ldim\HULL(\apfNEW) &\leq \LDI(\apfNEW) +\underline{\Delta}(\apfNEW).
	\end{split}
	\end{equation}
\end{theorem}
\begin{proof}
	Let $\varepsilon>0$. We will show that for all $\apfNEW(\cdot + \bt)$, $\bt \in \BR^{n}$ there exists $\overline{\bt} \in [0,l_{\apfNEW}(\varepsilon)]^{n}$ such that 
	\begin{equation}
	\|\apfNEW(\cdot+\bt)-\apfNEW(\cdot+\overline{\bt})\|_{\infty} \leq \varepsilon.
	\end{equation}
	Indeed, there is an $\varepsilon$-almost period $\tau \in [-\bt, -\bt + l_{\apfNEW}(\varepsilon)]^{n}$ for $\apfNEW(\cdot)$. Then $\overline{\bt}:=\bt+\tau$ is what we wanted. Now for arbitrary $\apfh \in \HULL(\apfNEW)$ there exists $\bt \in \BR^{n}$ such that $\|\apfh(\cdot)-\apfNEW(\cdot+\bt)\|_{\infty} \leq \varepsilon$ and, consequently,
	\begin{equation}
	\label{EQ: HelpfulEQ}
	\| \apfh(\cdot) - \apfNEW(\cdot+\overline{\bt}) \|_{\infty} \leq 2\varepsilon.
	\end{equation}
		
	For convenience' sake if $\mathcal{Q} \subset \BR^{n}$ let $\mathcal{Q}_{\apfNEW} := \{ \apfNEW(\cdot + \bt) \ | \ \bt \in \mathcal{Q} \} \subset \HULL(\apfNEW)$. It follows from \eqref{EQ: HelpfulEQ} that it is sufficient to cover the set $[0,l_{\apfNEW}(\varepsilon)]^{n}_{\apfNEW}$ by open balls. Let $\mathcal{B}_{\varepsilon}(\apfNEW(\cdot + \bt))$ be the open ball centered at $\apfNEW(\cdot + \bt)$ with radius $\varepsilon$. It is clear that for $\bt=(t_{1},\ldots,t_{n})$
	\begin{equation}
	\mathcal{B}_{\varepsilon}(\apfNEW(\cdot + \bt)) \supset \left(\prod\limits_{j=1}^{n} \left[t_{j}-\frac{\delta^{*}(\varepsilon)}{2}, t_{j} + \frac{\delta^{*}(\varepsilon)}{2}\right]\right)_{\apfNEW}.
	\end{equation}
	Thus, the set $[0,l_{\apfNEW}(\varepsilon)]^{n}_{\apfNEW}$ can be covered by $\left(\frac{l_{\apfNEW}(\varepsilon)}{\delta^{*}(\varepsilon)}+1\right)^{n}$ open balls of radius $\varepsilon$ and, consequently, the set $\HULL(\apfNEW)$ can be covered by the same number of balls of radius $3\varepsilon$. Therefore, $N_{3\varepsilon}(\HULL(\apfNEW)) \leq \left(\frac{l_{\apfNEW}(\varepsilon)}{\delta^{*}(\varepsilon)}+1\right)^{n}$ and
	\begin{equation}
	\label{EQ: LowerBeforeLimit}
	\frac{\ln N_{3\varepsilon}(\HULL(\apfNEW))}{\ln (1/\varepsilon)} \leq n\frac{\ln\left(\frac{l_{\apfNEW}(\varepsilon)}{\delta^{*}(\varepsilon)} + 1\right)}{\ln(1/\varepsilon)}.
	\end{equation}
	Taking it to the lower/upper limit in \eqref{EQ: LowerBeforeLimit} we finish the proof.
\end{proof}

It is easy to show that if $\chi(\cdot)$ satisfies a local H\"older condition with an exponent $\alpha \in (0,1]$ then $\udim(\chi(\MX)) \leq \frac{\udim\MX}{\alpha}$ and $\ldim(\chi(\MX)) \leq \frac{\ldim\MX}{\alpha}$.

Consider $\pi_{\MX} \colon \HULL(\apfNEW) \to \MX$ defined by $\pi_{\MX}(\apfh):=\apfh(0)$ for $\apfh \in \HULL(\apfNEW)$. It is clear that $\pi_{\MX}$ is a Lipschitz map. From Theorem \ref{TH: Bochner TH} it follows that $\pi_{\MX}(\HULL(\apfNEW))=\MO(\apfNEW)$. Thus, $\ldim\MO(\apfNEW) \leq \ldim\HULL(\apfNEW)$. Now we can prove theorem \ref{TH: MainResMatrixTh}.
\begin{proof}[Proof of theorem \ref{TH: MainResMatrixTh}]
	Using proposition \ref{PROP: ConttoDiscrete} we transit to continuous problem \eqref{EQ: MatrixKroneckerSysMainRes}. Proposition \ref{PROP: LinearityDimensionCoincide} with Theorem \ref{TH: LowerEstimate} applied to $\apfNEW(\bt):=\hat{A}\bt$ give the desired result: it is quite clear that $\apfNEW$ is Lipschitz continuous and, consequently, $\underline{\Delta}(\apfNEW) \leq 1$. Now we use $\ldim\MO(\apfNEW) \leq \ldim\HULL(\apfNEW)$ and the second inequality in \eqref{EQ: MainTheoremofDioDim}.
\end{proof}
\section{An upper estimate via Diophantine approximations}
One of the basic properties of the Diophantine dimension is given by the following simple lemma (see \cite{Anikushin2017}).
\begin{lemma}
	\label{LEM: DioHolder}
	Let $\apfNEW \colon \BR^{n} \to \MX$ be almost periodic and let $\chi \colon \MX \to \MY$ satisfy a local H\"older condition with an exponent $\alpha \in (0,1]$; then
	\begin{equation}
	\begin{split}
	\UDI(\chi \circ \apfNEW) &\leq \frac{\UDI(\apfNEW)}{\alpha},\\
	\LDI(\chi \circ \apfNEW) &\leq \frac{\LDI(\apfNEW)}{\alpha}.
	\end{split}
	\end{equation}
\end{lemma}

\begin{remark}
	\label{REM: RemarkHelpfull}
	Let the function $\Phi \colon \BT^{m+1} \to \MX$, where $\MX$ is a complete metric space, satisfy a local H\"older condition with an exponent $\alpha \in (0,1]$ and let $\omega_{1},\ldots,\omega_{m}$ be real numbers. The function $\apfNEW(t):=\Phi(t,\omega_{1}t,\ldots,\omega_{m}t)$, $t \in \BR$, is almost periodic as the image of an almost periodic function under a uniformly continuous map. By Lemma \ref{LEM: DioHolder}, $\UDI(\apfNEW) \leq \frac{\UDI(v)}{\alpha}$, where $v(t):=(t,\omega_{1},\ldots,\omega_{m}t)$ is a linear flow on $\BT^{m+1}$. So it is sufficient to estimate $\UDI(v)$.
\end{remark}

Since there are no algorithms (similar to the classical continued fraction expansion), which could provide a sequence of convergents $\{q_{k}\}$ with <<good>> properties, for the simultaneous approximation case, we prove the existence of convergents with required properties. At first we need the classical Dirichlet theorem.
\begin{theorem}
	\label{TH: DirichletTheorem}
	Let $\omega=(\omega_{1},\ldots,\omega_{m})$ be an $m$-tuple of real numbers; then for every $Q>0$ there is $1 < q < Q$ such that
	\begin{equation}
	|\omega q|_{m} < \left(\frac{1}{Q}\right)^{\frac{1}{m}}.
	\end{equation}
\end{theorem}
\noindent  The following lemma directly follows from Theorem \ref{TH: DirichletTheorem}.
\begin{lemma}
	\label{TH: DiophantineMultiCondition}
	Let an $m$-tuple $\omega=(\omega_{1},\ldots,\omega_{m})$ satisfy the Diophantine condition of order $\nu \geq 0$; then there are a non-decreasing sequence of natural numbers $\{q_{k}\}$, $k=1,2,\ldots$, and a constant $\hat{C}=\hat{C}(\omega)>0$ such that
	\begin{enumerate}
		\item[(A1)] $|\omega q_{k}|_{m} \leq \hat{C} \cdot
		\left(\frac{1}{q_{k+1}}\right)^{1/m}$.
		\item[(A2)] $q_{k+1}=O\left(q^{1+\nu}_{k}\right)$ and for $a_{k+1}:=\lfloor\frac{q_{k+1}}{q_{k}}\rfloor$ we have $a_{k+1}=O(q^{\nu}_{k})$. Also there are constants $\gamma_{2}>\gamma_{1}>1$ and $A_{1},A_{2}>0$ such that
		\begin{equation}
		\label{EQ: GeometricGrowthMultiConvergents}
		A_{1} \gamma^{k}_{1}\leq q_{k} \leq A_{2}\gamma^{k}_{2}.
		\end{equation}
		\item[(A3)] For every $\eta>0$ there exists $C_{\eta}>0$ such that the estimate $\sum\limits_{k=N}^{\infty}\left(\frac{1}{q_{k}}\right)^{\eta} \leq
		C_{\eta}\frac{N}{q^{\eta}_{N}}$ holds.
	\end{enumerate}
\end{lemma}
\begin{proof}
	Despite the fact that $(A3)$ is the direct corollary of $(A2)$ and $(A2)$ is <<almost>> follows from $(A1)$ we need these assumptions in such a formulation for the convenience.
	
	By the Dirichlet theorem, for any $Q>0$ there is a natural number $1 \leq q \leq Q$ such that
	\begin{equation}
	\label{EQ: DirichletMultiConvergents}
	|\omega q|_{m} \leq \left(\frac{1}{Q}\right)^{1/m}.
	\end{equation}
	Let $\beta>1$ be a fixed real number and let
	$q_{k}$, $k=1,2,\ldots$, be a natural $q$ from the Dirichlet theorem for $Q = \beta^{k}$. If for some $k$ we have $q_{k+1}<q_{k}$, then we put
	$q_{k}:=q_{k+1}$ and repeat such process for smaller $k$. Note that $q_{k} \to +\infty$ as $k \to \infty$ (as, by the Diophantine condition, there is at least one irrational $\omega_{j}$) and this guarantees that for every $k$ the value of $q_{k}$ will be changed only for a finite number of times. Thus, we have a non-decreasing sequence $\{q_{k}\}$, $k=1,2,\ldots$, where $q=q_{k}$ satisfies \eqref{EQ: DirichletMultiConvergents} for $Q=\beta^{k}$.
	
	Now from the Diophantine condition we have
	\begin{equation}
	C_{d}\left(\frac{1}{q^{1+\nu}_{k}}\right)^{1/m} \leq |\omega q_{k} |_{m} < \left(\frac{1}{\beta^{k}}\right)^{1/m} \leq \left(\frac{1}{q_{k}}\right)^{1/m}
	\end{equation}
	and, consequently,
	\begin{equation}
	C_{d}^{\frac{m}{1+\nu}}\beta^{\frac{k}{1+\nu}} \leq q_{k} \leq
	\beta^{k}.
	\end{equation}
	It is easy to see that
	\begin{equation}
	q_{k+1} \leq
	\beta^{k+1}=\beta\cdot\beta^{k}=\beta
	C_{d}^{-m}\left(C^{\frac{m}{1+\nu}}
	\beta^{\frac{k}{1+\nu}}\right)^{1+\nu}\leq \beta C_{d}^{-m}
	q^{1+\nu}_{k}.
	\end{equation}
	Therefore, $q_{k+1}=O(q^{1+\nu}_{k})$ and, it is obvious that $a_{k+1}=O(q^{\nu}_{k})$. Now put
	$\gamma_{1}:=\beta^{\frac{1}{1+\nu}}$ with
	$A_{1}:=C_{d}^{\frac{m}{1+\nu}}$ and
	$\gamma_{2}:=\beta$ with $A_{2}:=1$. Also note that
	\begin{equation}
	|\omega q_{k}| <
	\left(\frac{1}{\beta^{k}}\right)^{1/m}=\beta^{1/m}\left(\frac{1}{\beta^{k+1}}\right)^{1/m}
	\leq \beta^{1/m}\left(\frac{1}{q_{k+1}}\right)^{1/m}.
	\end{equation}
	
	Now let's estimate
	$\sum\limits_{k=N}^{\infty}\left(\frac{1}{q_{k}}\right)^{\eta}=\frac{1}{q^{\eta}_{N}}\sum\limits_{k=0}^{\infty}\left(\frac{q_{N}}{q_{N+k}}\right)^{\eta}.$
	From \eqref{EQ: GeometricGrowthMultiConvergents} we get
	\begin{equation}
	\frac{q_{N}}{q_{N+k}} \leq
	\frac{A_2}{A_1}\frac{\gamma^{N}_{2}}{\gamma^{N+k}_{1}}=\frac{A_2}{A_1}\beta^{\frac{N\nu
			- k}{1+\nu}}.
	\end{equation}
	We use this estimate only for $k \geq 2\nu N$, and for others $k$ we just use $\frac{q_{N}}{q_{N+k}} \leq 1$. Thus,
	\begin{equation}
	\frac{1}{q^{\eta}_{N}}\sum\limits_{k=0}^{\infty}\left(\frac{q_{N}}{q_{N+k}}\right)^{\eta} \leq
	\frac{1}{q^{\eta}_{N}} 2\nu N + \left(\frac{A_2}{A_1}\right)^{\eta}\frac{1}{q^{\eta}_{N}}\sum\limits_{k \geq 2\nu N}\beta^{\eta\frac{N\nu - k}{1+\nu}} \leq C_{\eta}\cdot\frac{N}{q^{\eta}_{N}},
	\end{equation}
	where $C_{\eta}$ is an appropriate constant.
\end{proof}

Now we are ready to prove the following
\begin{theorem}
	\label{TH: UpperDiophantineEstimateTheoremMULTI}
	Let $\apfNEW(t)=\Phi(t,\omega_{1} t,\ldots,\omega_{m}t)$ be an almost periodic function, where $\Phi \colon \BT^{m+1} \to \MX$ satisfies a local H\"older condition with an exponent $\alpha \in (0,1]$. Let the $m$-tuple $\omega=(\omega_{1},\ldots,\omega_{m})$ satisfy the Diophantine condition of order $\nu \geq 0$ with $\nu(m-1) < 1$; then
	\begin{equation}
	\label{EQ: UpperDiophantineEstimateMULTI}
	\UDI(\apfNEW) \leq \frac{1}{\alpha}\cdot\frac{(1+\nu)m}{1-\nu(m-1)}.
	\end{equation}
\end{theorem}
\begin{proof}
	From Remark \ref{REM: RemarkHelpfull} it is sufficient to estimate the Diophantine dimension of the linear flow $v(t)=(t,\omega_{1}t,\ldots,\omega_{m}t)$ on $\BT^{m+1}$.
	
	Let $\{q_{k}\}$, $k=1,2,\ldots$, be the sequence of simultaneous denominators (=convergents) provided by Lemma  \ref{TH: DiophantineMultiCondition} for $\omega$. Put $a_{k+1}:=\lfloor \frac{q_{k+1}}{q_{k}} \rfloor \geq 1$. Then $a_{k+1}q_{k} \leq q_{k+1} \leq (a_{k+1}+1)q_{k}$ and due to (A1) we have
	\begin{equation}
	|\omega q_{k}|_{m} \leq \hat{C} \left(\frac{1}{a_{k+1}q_{k}}\right)^{1/m}.
	\end{equation}
	
	Now let $k_{0}$ be a sufficiently large number. We will show that for every $A \in \BR$ there is $\tau$ such that $|\tau-A| \leq q_{k_{0}}$. Firstly, suppose $A \geq q_{k_{0}}$. Let $K$ be a number such that $q_{K} \leq A$ and $q_{K+1} > A$. We put $\tau = \tau(A) = \sum\limits_{k=k_{0}}^{K}p_{k} q_{k}$, where $p_{k} \geq 0$ and $p_{k} \in \BZ$ is constructed by the following procedure. Let $p_{K} \geq 0$ be an integer such that $p_{K}q_{K} \leq A$ and $(p_{K}+1)q_{K} > A$. It is clear that $p_{K} \leq a_{K+1}$. Now let $p_{K-1} \geq 0$ be such that $p_{K-1}q_{K-1} + p_{K} q_{K} \leq A$ and $(p_{K-1}+1)q_{K-1} + p_{K}q_{K} > A$. We continue such a procedure to get a sequence of integer numbers $p_{k_{0}},\ldots,p_{K}$, where $0 \leq p_{k} \leq a_{k+1}$, $k=k_{0},\ldots,K$. By definition $|\tau - A|=A-\tau \leq q_{k_{0}}$. Now put $\tau(A):=\tau(-A)$ for $A \leq -q_{k_{0}}$ and $\tau(A):=0$ for $-q_{k_{0}}< A < q_{k_{0}}$. Thus, for every $A \in \BR$ there is $\tau$ such that $|\tau-A| \leq q_{k_{0}}$.
	
	From (A2), (A3) and from the fact that $p_{k} \leq a_{k+1}$ and $a_{k+1} = O(q^{\nu}_{k})$ we have
	\begin{equation}
	\label{EQ: EstimateAPLinearFlow}
	|\omega\tau|_{m} \leq \hat{C}\sum\limits_{k=k_0}^{K}\left(\frac{p^{m}_{k}}{a_{k+1}q_{k}}\right)^{1/m} \leq C_{1}\sum\limits_{k=k_0}^{K}\frac{1}{q^{\eta}_{k}} \leq C_{2}\frac{k_0}{q^{\eta}_{k_0}},
	\end{equation}
	where $\eta = \frac{1 - \nu(m-1)}{m}$ and $C_1,C_2>0$ are appropriate constants.
	Let $\varepsilon_{k}:=C_{2}\frac{k}{q^{\eta}_{k}}$. We showed that the value $L_{k}:=q_{k}$ is an upper bound for the inclusion length $l_{v}(\varepsilon_{k})$ of $\varepsilon_{k}$-almost periods of $v(\cdot)$. 
	
	Now for all sufficiently small $\varepsilon>0$ such that $\varepsilon_{k+1} < \varepsilon \leq \varepsilon_{k}$ put $L(\varepsilon):=L_{k+1}=q_{k+1}$, which is an upper bound for $l_{v}(\varepsilon)$. Note that for all sufficiently small $\delta>0$ and for large enough $k$ the inequality $\left(\frac{1}{\varepsilon_{k}}\right)^{\eta^{-1}} \geq q^{1-\delta}_{k}$ holds. For some constant $C_{3}$ we have
	\begin{equation}
	L(\varepsilon)=q_{k+1} \leq C_{3}(q_{k})^{1+\nu} \leq C_{3}\left(\frac{1}{\varepsilon_{k}}\right)^{\frac{1+\nu}{\eta(1-\delta)}} \leq C_{3}\left(\frac{1}{\varepsilon}\right)^{\frac{1+\nu}{\eta (1-\delta)}}.
	\end{equation}
	In particular, $\UDI(v) \leq \frac{1+\nu}{\eta (1-\delta)}m$. Taking $\delta$ to zero we have that $\UDI(v) \leq \frac{1+\nu}{\eta} = \frac{(1+\nu)m}{1-\nu(m-1)}$. Thus, the theorem is proved.
\end{proof}
\begin{remark}
\label{REM: MultiPhen}
The restriction $\nu (m-1) < 1$ in Theorem \ref{TH: UpperDiophantineEstimateTheoremMULTI}  is similar to the one in \cite{Moser1990} (see Theorem 2 therein). But for our case we don't know are there Diophantine $m$-tuples $(\omega_{1},\ldots,\omega_{m})$ with $\UDI(v)=\infty$ for $v(t)=(t,\omega_{1}t,\ldots,\omega_{m}t)$.
\end{remark}

The proof of Theorem \ref{TH: MainResUpperEstimate} is as follows.
\begin{proof}[Proof of Theorem \ref{TH: MainResUpperEstimate}]
	Using Propositions \ref{PROP: LinearityDimensionCoincide} and \ref{PROP: ConttoDiscrete} with Theorem \ref{TH: UpperDiophantineEstimateTheoremMULTI} applied to $\apfNEW(t):=(t,\omega_{1} t,\ldots,\omega_{m} t)$ we get the desired result.
\end{proof}
\section{Discussing}
Let $\mathcal{D}_{m}(\nu)$ be the set of all $m$-tuples satisfying the Diophantine condition of order $\nu \geq 0$. Put $\Omega'_{m} = \bigcap\limits_{\nu > 0} \mathcal{D}_{m}(\nu) \cup \mathcal{D}_{m}(0)$. Since every $\mathcal{D}_{m}(\nu)$ for $\nu > 0$ is a set of full measure (see \cite{Kleinblock1998}) and $\mathcal{D}_{m}(\nu_{1}) \supset \mathcal{D}_{m}(\nu_{2})$ for $\nu_{1} < \nu_{2}$, the set $\Omega'_{m}$ is a set of full measure. Now let $\Omega_{m}$ be the set of linearly independent $m$-tuples in $\Omega'_{m}$.
\begin{proof}[Proof of Corollary \ref{COR: ExactCalculation}]
	Since $1,\omega_{1},\ldots,\omega_{m}$ are linearly independent we have $\MO(\apff) = \BT^{m+1}$ and, thus, $\ldim(\MO(\apff)) = m+1$ and, by Theorem \ref{TH: MainResMatrixTh}, $\LDI(\mathfrak{K}) \geq m$.
	
	Consider the estimate $\UDI(\mathfrak{K}) \leq \frac{(1+\nu)m}{1-\nu(m-1)}$ given by Theorem \ref{TH: MainResUpperEstimate}. For $\omega \in \mathcal{D}_{m}(0)$ we immediately get what we need. If $\omega \in \Omega_{m}$, i. e. $\omega \in \mathcal{D}_{m}(\nu)$ for all $\nu > 0$ then one should take the limit as $\nu \to 0+$ in the above estimate.
\end{proof}
Now, within assumptions of corollary \ref{COR: ExactCalculation}, for $\theta \in \BT^{m}$ and $\varepsilon>0$ consider the classical Kronecker system
\begin{equation}
\label{EQ: ClassicalKroneckerSys}
|\omega q - \theta|_{m} \leq \varepsilon.
\end{equation}
As corollary \ref{COR: ExactCalculation} state, for every $\delta>0$ there is an $\varepsilon_{0}>0$ such that every segment $[a, L^{+}(\varepsilon)]$, with $a \in \BR$ and $L^{+}(\varepsilon)=\left(\frac{1}{\varepsilon}\right)^{m+\delta}$, contains an integer solution $q$ to \eqref{EQ: ClassicalKroneckerSys} with $\varepsilon \leq \varepsilon_{0}$ and there is a segment $[a(\varepsilon),L^{-}(\varepsilon)]$ with $L^{-}(\varepsilon)=\left(\frac{1}{\varepsilon}\right)^{m-\delta}$ and with no integer solutions. In particular, there is a solution $q \in \BZ$ with $|q| \leq \left(\frac{1}{\varepsilon}\right)^{m+\delta}$. The latter asymptotic is well-known for the case, when $\omega_{1},\ldots,\omega_{m}$ are algebraic numbers (see remark 3.1 in \cite{fukshansky2018effective}). Note that such algebraic $m$-tuples are contained in $\Omega_{m}$ due to Schmidt's subspace theorem (see \cite{Schmidt1980}).

If the numbers $\omega_{1},\ldots,\omega_{m}$ satisfy the Diophantine condition of order $\nu$ and $\nu(m-1) < 1$ then, by theorem \ref{TH: UpperDiophantineEstimateTheoremMULTI}, for any $\delta>0$ and sufficiently small $\varepsilon$, every segment $[0,L^{+}(\varepsilon)]$, where $L^{+}(\varepsilon)=\left(\frac{1}{\varepsilon}\right)^{\frac{(1+\nu)m}{1-\nu(m-1)} + \delta}$, contains an integer solution $q$ to the system \eqref{EQ: ClassicalKroneckerSys}.

Now we will discuss how such properties affect the dynamics of almost periodic trajectories. For example, let $\apfNEW(t)=e^{i2\pi t} + e^{i2\pi \omega t}$, where $\omega$ is an irrational number. It is clear that $\MO(\apfNEW)$ is the disk of radius 2. Let $\frac{p_{k}}{q_{k}}$, $k=1,2,\ldots,$ be the sequence of convergents given by the continued fraction expansion of $\omega$. So, $\omega \approx \frac{p_{k}}{q_{k}}$ and $\apfNEW(t)$ is close to $q_{k}$-periodic trajectory $\varphi_{k}(t):=e^{i2\pi t} + e^{i2\pi \frac{p_{k}}{q_{k}} t}$ for some time interval.
\begin{figure}[h]
	\begin{minipage}[h]{0.49\linewidth}
		\center{\includegraphics[width=\linewidth]{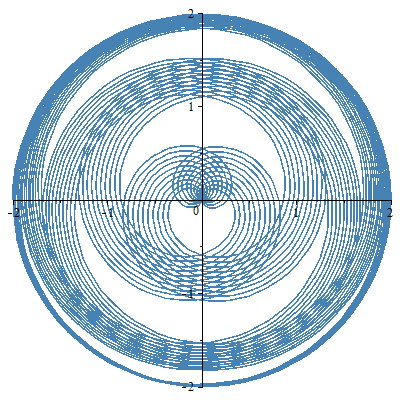} \\ a) $\omega=\zeta(3), t=0..35$.}
	\end{minipage}
	\hfill
	\begin{minipage}[h]{0.49\linewidth}
		\center{\includegraphics[width=\linewidth]{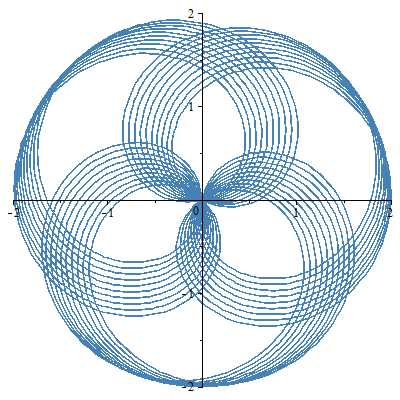} \\ b) $\omega=\pi^{\zeta(3)}, t=0..10$.}
	\end{minipage}
	\caption{A piece of the trajectory of $\apfNEW(t)$ for certain $\omega$'s.}
	\label{PIC: AnimatedBehaviour2}
\end{figure}
The length of such an interval depends on how good the fraction $\frac{p_{k}}{q_{k}}$ approximates $\omega$. The latter depends on the growth rate of $q_{k}$. So, if $\omega$ is well-approximable, namely $q_{k}$ grows sufficiently fast, then $\apfNEW(\cdot)$ is similar to a periodic trajectory, during a large, in comparison to the period, time interval. As a result, in many cases the trajectory fills the disk $\MO(\apfNEW)$ in a very lazy manner (see Fig. \ref{PIC: AnimatedBehaviour2}). On the other hand, for a badly approximable $\omega$ the filling is more uniform. Note that the trajectory of $\apfNEW$ is uniformly distributed with respect to a probability measure $\mu$, independent of $\omega$ (see \cite{Anikushin2017}). The latter means that for all Borel subsets $\mathcal{C} \subset \MO(\apfNEW)$ we have
\begin{equation*}
\lim\limits_{T \to +\infty} \frac{1}{2T}\int_{-T}^{T}\mathbf{1}_{\mathcal{C}}(\apfNEW(t))dt = \mu(C).
\end{equation*}
Thus, such arithmetic properties of $\omega$ affect a character of evolution and not the asymptotic distribution of $\apfNEW$.

Approximation theorem for almost periodic functions with a similar reasoning extend such phenomena to the case of general almost periodic functions. As well as the measure of irrationality of $\omega$ provides a quantitative information about the dynamic behaviour in the simple case considered above, the Diophantine dimension does this for general almost periodic functions.

\section*{Acknowledgements} This work is supported by the German-Russian Interdisciplinary Science Center (G-RISC) funded by the German Federal Foreign Office via the German Academic Exchange Service (DAAD): Projects M-2017a-5 and M-2017b-9.

\bibliographystyle{amsplain}

\end{document}